\documentclass[a4paper,11pt,reqno]{amsart}

\usepackage{amssymb}
\usepackage{amsfonts}
\usepackage{amsmath}
\usepackage{amscd}
\usepackage{amsthm}
\usepackage{color}
\usepackage{mathtools}
\usepackage{tikz-cd}
\usepackage[pagebackref, hypertexnames=false, colorlinks, citecolor=red, linkcolor=blue, urlcolor=red]{hyperref}

\theoremstyle{plain}

\newtheorem{thm}{Theorem}[section]
\newtheorem*{thm*}{Theorem}
\newtheorem{cor}[thm]{Corollary}
\newtheorem{lem}[thm]{Lemma}
\newtheorem{prop}[thm]{Proposition}

\newtheorem{rem}[thm]{Remark}

\newtheoremstyle{named}{}{}{\itshape}{}{\bfseries}{.}{.5em}{#1 \thmnote{#3}}
\theoremstyle{named}

\numberwithin{equation}{section}

\DeclareSymbolFont{rsfs}{U}{rsfs}{m}{n}
\DeclareSymbolFontAlphabet{\mathscr}{rsfs}

\def\VV{{\mathcal V}}
\def\C{{\mathbb C}}
\def\T{{\mathbb T}}
\def\P{{\mathcal P}}

\def\ra{\rightarrow}

\def\beq{\begin{eqnarray}}
\def\eeq{\end{eqnarray}}
\def\beqa{\begin{eqnarray*}}
\def\eeqa{\end{eqnarray*}}

\def\bl{\boldsymbol}

\def\ra{\rightarrow}

\begin{document}

\title[Approximation on Rational Curves] {Approximation in the mean on rational curves}

\author{Shibananda Biswas}
\address[S.~Biswas]{Department of Mathematics and Statistics, Indian Institute of Science Education and Research Kolkata, Mohanpur 741246, Nadia, West Bengal, India}
\email{\tt shibananda@iiserkol.ac.in}

\author{Mihai Putinar}
\address[M.~Putinar]{University of California at Santa Barbara, CA,
USA and Newcastle University, Newcastle upon Tyne, UK} 
\email{\tt mputinar@math.ucsb.edu, mihai.putinar@ncl.ac.uk}

\thanks{}
\dedicatory{J\"org Eschmeier, in memoriam}

\subjclass[2010]{41A10, 41A20, 47B20, 14H45} \keywords{Bounded point evaluation, rational approximation, subnormal operator, rational curve}
\date  {}

\begin{abstract} In the presence of a positive, compactly supported measure on an affine algebraic curve, we relate
the density of polynomials in Lebesgue $L^2$-space to the existence of analytic bounded point evaluations.
Analogues to the complex plane results of Thomson and Brennan are obtained on rational curves.

\end{abstract}

\maketitle
\section{Introduction}

Let $n$ be a positive integer and $\C[z]$ denote the algebra of polynomials in $n$ complex variables $z = (z_1, z_2, \ldots, z_n).$
Let $\VV \subset \C^n$ be a complex affine curve, that is, the common zero set of a finite system of polynomials. When necessary, we consider $\VV$ as an algebraic variety endowed with its reduced structural sheaf.
Given a positive Borel measure $\mu$ supported by a compact subset $K$ of $\VV$ we consider the closure $P^2(\mu)$ of $\C[z]$ in
$L^2(\mu)$. A central question of function theory is the relationship between the density of polynomials in Lebesgue space, that is $P^2(\mu) = L^2(\mu)$ versus the existence of {\it bounded point evaluations} $\lambda$ for $P^2(\mu)$:
$$ |p(\lambda)| \leq C \|p \|_{2,\mu}, \ \ \ p \in \C[z].$$
Note that, if such a point exists, then it belongs to $\VV$. If the above estimate holds, with the same constant, for an open set, we say that the measure $\mu$ admits {\it analytic bounded point evaluations}. By enlarging the notion of (analytic) bounded point evaluation beyond polynomials, such as to rational functions with prescribed pole location, one has to specify the algebra of analytic functions for which the above bound holds.

In the case of measures defined on the complex plane, this density problem is classical, naturally  related to Szeg\"o's Limit Theorem (on the circle), determinateness of the moment problem on the line, the structure of cyclic subnormal operators \cite{K,Akh,C}. We owe to Jim Thomson \cite{T} the definitive answer, encrypted in one definitive statement which culminates more than half a century of partial results: \bigskip

{\it For a positive Borel measure $\mu$, compactly supported on the complex plane and without point masses,
$P^2(\mu) \neq L^2(\mu)$ if and only if there exist analytic bounded point evaluations.}\bigskip

A different proof of Thomson's theorem appears in \cite{Brennan-2005}, together with authoritative historical comments. The cloud of such point evaluations gives essential information about the building blocks of subnormal operators \cite{C}. Several generalizations of Thomson's Theorem to the case of rational functions have been proposed, see the recent survey \cite{CY} for a detailed account, or \cite{Yang} for some recent advances. In the present note we relay on Brennan's rational approximation theorem, reported with an original proof in \cite{Brennan-2008}.

Our aim is to start investigating conditions assuring the validity of Thomson's Theorem on algebraic curves. 
\section{Algebraic Curves with Polynomial Parametrization}

To fix ideas we start with the simplest framework. In this section we assume that the algebraic curve $\VV$ admits a polynomial parametrization. This is a well studied subclass of rational curves, not 
neglected by its relevance to numerical algebraic geometry \cite{SWP}. To simplify terminology, we say that $\VV$ is a {\it polynomial curve}.

A key algebraic observation is that one can change the parametrization of $\VV$ to a proper polynomial parametrization \cite[Theorem 6.11]{SWP}, which in turn is normal \cite[Corollary 6.21]{SWP}. To be more specific, given a polynomial curve $\VV$, there exists a map 
$$P = (p_1,\ldots,p_n): \C\ra \VV, \ \ p_i\in\C[\zeta], 1\leq i\leq n,$$ 
with the property that $P$ is bijective away from a finite set of points and avoids at most finitely many points of $\VV$. The complex coordinate on the parameter space is $\zeta \in \C$.
That is the fibre $P^{-1}P(\zeta)$ has cardinality one, with the exception of finitely many points $ \zeta \in \C$, where it is finite.

We denote by $\mathcal O$ the sheaf of complex analytic functions.  Since the map $P$ is finite, Grauert's Theorem states that the direct image sheaf $P_\ast {\mathcal O_{\C}}$, defined by sheaf associated to the presheaf $P_\ast {\mathcal O_{\C}}(U) = \mathcal O(P^{-1}(U))$ for $U$ open in $\C^n$, is coherent as 
a ${{\mathcal O}_{{\C}^n}}$-module. We refer to \cite[Chapter I, Section 3]{GRtss} for the proof and terminology. Moreover, Theorem I.1.5 in \cite{GRtss} asserts:
\begin{eqnarray}\label{DI}
P_\ast {\mathcal O_{\C}}({\C}^n) = {\mathcal O}({\C}).
\end{eqnarray}
Consider the natural pull-back morphism
$$ \Psi: {\mathcal O}_{{\C}^n} \longrightarrow P_* {\mathcal O}_{\C}$$
defined by $h \mapsto P^* h = h \circ P.$

\begin{lem}\label{finite}
Let $P = (p_1,\ldots,p_n): \C\ra \C^n$, $p_i\in\C[z], 1\leq i\leq n,$ be an injective map away from a finite set of points.  Then
$$ \dim {\mathcal O}({\C})/ P^\ast {\mathcal O}({{\C}^n}) < \infty.$$
\end{lem}
\begin{proof}
The kernel $\mathscr I$ of $\Psi$ is the ideal sheaf defining the curve $\mathcal V = P({\C}).$ The short exact sequence 
$$
\begin{tikzcd}
  0 \arrow[r] & \mathscr{I} \arrow[hookrightarrow]{r} & \mathcal O_{\C^n}\arrow{r}{\psi}& \mathfrak{Im}\Psi\ar{r} & 0
\end{tikzcd}
$$
implies the coherence of the image sheaf $\mathfrak{Im}\Psi$. In view of the remark I.1.5 of \cite[page 48]{GRtss}, the cokernel $\mathscr S$ of $\Psi$
is zero at every point $y \in {\C}^n$ with ${\rm card} (P^{-1}(y)) =1$. Hence $\mathscr S$ is a coherent analytic sheaf, supported by finitely many
points of ${\C}^n$, that is
$$ \dim {\mathscr S}({\C}^n) < \infty.$$
The long exact sequence of cohomology induced from the short exact sequence 
$$
\begin{tikzcd}
  0 \arrow[r] & \mathfrak{Im}\Psi \arrow[hookrightarrow]{r} & P_* {\mathcal O}_{\C}\arrow{r}{q}& \mathscr S\ar{r} & 0,
\end{tikzcd}
$$
where $q$ is the quotient morphism, implies $H^1(\C^n, \mathfrak{Im}\Psi) =0.$ Hence, from equation \ref{DI} the desired finiteness follows.
\end{proof}

Returning to the polynomial density in Lebesgue space, we carry the data from $\VV$ to the parameter space, and there we invoke Thomson's Theorem.
More precisely, let $\mu$ be a positive Borel measure supported by a compact subset $K$ of $\VV$. Since we are seeking non-trivial bounded point evaluations, we assume that $\mu$ doers not have point masses, that is $\mu(\{ y \}) = 0$ for all $y \in K$.
The pre-image set $L = P^{-1}(K)$ is compact since the map $P$ is finite.

Let $A \subset K$ denote the finite subset of points $y \in K$ with the cardinality of $P^{-1}(\{y\})$ bigger than one. The singular points of $\VV$ are included in $A$.
Let $B = P^{-1} (A)$, also a finite subset of $L$. Denote $K' = K \setminus A$ and $L' = L \setminus B$. To the extent that the restriction map
$$ P' = P|_{L'} : L' \longrightarrow K'$$ is bijective. Finally, let $\mu' = \chi_{K'} \mu$ denote the restriction of the measure $\mu$ to $K'$, and take the push-forward measure $\nu = (P'^{-1})_\ast \mu'$. The positive measure $\nu$ is supported by $\VV$ and does not possess atoms.

Since the measure $\mu$ does not carry point masses, $L^2(\mu) = L^2(\mu')$ isometrically, and consequently $P^2(\mu) = P^2(\mu')$. Let $\phi$ be a continuous function on $\C^n$. Then
$$ \int_K \phi  d\mu =  \int_{K'} \phi d\mu' = \int \phi \circ P d\nu.$$
In other terms the pull-back map
$$ P^\ast : L^2(\mu) \longrightarrow L^2(\nu)$$
is isometric. Due to the local structure of an algebraic curve, continuous functions in the ambient space separate, modulo finitely many singular points, different branches of the curve, see for instance Section 2.5 in \cite{SWP}. Hence, by passing to Borel functions, we find that the isometric map $P^\ast$ is onto, hence a unitary operator.

\begin{prop}\label{BPE} Under the assumptions above,
the space $ P^2(\mu)$ admits analytic bounded point evaluations if and only if $P^2(\nu)$ admits analytic bounded point evaluations.
\end{prop}
\begin{proof}
Assume that $\alpha \in L$ is a analytic bounded point evaluation with respect to $P^2(\nu)$. Let $P(\alpha) = \beta$. For every $p\in \C[\bl z]$ one finds
$$
|p(\beta)| = |p\circ P (\alpha)|\leq C\|p\circ P\|_{2,\nu} = C\|p\|_{2,\mu}
$$
where $C>0$ is a universal constant. Therefore $ P^2(\mu)$ admits analytic point evaluations filling a neighborhood of $\beta$. 

Conversely, assume that $ P^2(\mu)$ admits bounded point evaluations (with the same bound $C$) at every point of an open subset $U$ of $\VV$. While the map $P : \C \longrightarrow \VV$ may not be surjective,
it avoids only finitely many points of $\VV$ in its range, cf. Theorem 2.2.43 in \cite{SWP}. Choose a point $\beta \in U$ which possesses a pre-image, that is
 $\alpha \in \C$ such that $P(\alpha) = \beta$. We claim that at $\alpha$, is a bounded point evaluation for  $P^2(\nu)$. The algebra of entire functions ${\mathcal O}(\C^n)$ is dense in $P^2(\mu)$, and we have deduced from Grauert's finiteness theorem that $P^\ast {\mathcal O}(\C^n)$ is a finite codimensional subspace of ${\mathcal O}(\C)$. Hence there exists a finite dimensional space of entire functions $W \subset {\mathcal O}(\C)$ with the property that 
 $P^\ast {\mathcal O}(\C^n) + W$ is a dense subspace of $P^2(\nu)$. Since $\alpha$ is a bounded point evaluation for both $P^\ast {\mathcal O}(\C^n)$ and $W$, and
 $\dim W<\infty$, we infer that $\alpha$ is a bounded point evaluation for the entire space $P^2(\nu)$, with a locally bounded constant $C(\alpha) = \sup\frac{p(\alpha)}{\| p \|_{2,\nu}}$.
\end{proof}

\begin{thm} \label{poly}Let $\VV \subset \C^n$ be a polynomial curve and let $\mu$ be a positive Borel measure supported by a compact subset of $\VV$. Assume that $\mu$ does not have point masses and $ P^2(\mu) \neq L^2(\mu)$. Then, and only then, there exist analytic bounded point evaluations for $ P^2(\mu)$.
\end{thm}
\begin{proof}
Choose, as before in this section, a normal, polynomial parametrization $P: \C \longrightarrow \C^n$ of the curve $\VV$. Let $\nu$ denote the pull-back measure of $\mu$ on $\C$. We claim $P^2(\nu) \neq L^2(\nu)$. 

Suppose by contradiction $P^2(\nu) = L^2(\nu)$. Lemma \ref{finite} implies that $P^\ast P^2(\mu)$ is a finite codimenion subspace of  $L^2(\nu)$. But the pull-back map $P^\ast$ is unitary at the level of $L^2$ spaces. Hence $P^2(\mu)$ is a finite codimension subspace of $L^2(\mu)$.

The multiplication operator $M_i$ on $P^2(\mu)$ by the coordinate function $z_i$ is subnormal, for every $i, \ 1 \leq i \leq n.$ The corresponding normal extension $N_i$, is represented by the multiplication by  $z_i$ on $L^2(\mu)$. With respect to the decomposition $L^2(\mu) = P^2(\mu)\oplus P^2(\mu)^\perp$, we can write $N_i$ in $2\times 2$ blocks:
$$
N_i = 
\begin{bmatrix}
M_i & S_i\\
0 & T_i
\end{bmatrix}
$$
with both $S_i$ and $T_i$ finite rank operators, $1\leq i\leq n$. The normality block operator equation of these extensions yields:
$$
[M_i^*, M_i] = S_i S_i^* \mathrm{~and~} S_i^*S_i = [T_i, T_i^*].
$$
Since $S_i$ is a finite rank operator and trace of $S_i^*S_i$ vanishes, we find $S_i = 0$ for all $i, 1\leq i\leq n$. This in turn shows that both $M_{z_i}$ and $T_i$'s are normal operators. By assumption $ P^2(\mu) \neq L^2(\mu)$, that is the block carrying the commuting normal matrices $T_i$ is non-trivial. Then a common eigenvector of the $T_i$'s exists, implying the existence of a point mass for the measure $\mu$. A contradiction.

According to Thomson's theorem \cite{T}, the space $P^2(\nu)$ admits a non-empty open set of bounded point evaluations. Hence by Proposition, \ref{BPE},  $ P^2(\mu)$ has analytic bounded point evaluations. For the only then part, we note that if  $ P^2(\mu)$ admits bounded point evaluations, then $P^2(\mu) \neq L^2(\mu)$ as other wise the unitarity of the pull back map implies $P^2(\nu) = L^2(\nu)$ which is a contradiction to Thomson's theorem via Proposition, \ref{BPE}. This completes the proof.
\end{proof}

\section{Rational curves} The general case of rational curves is not much different. This time we change the base via a rational map, and invoke a generalization of Thomson's theorem proved by Brennan \cite{Brennan-2008}. We state the main result and indicate the very similar deduction molded on the proof detailed in the previous section.

\begin{thm}\label{rat} Let $\VV$ be a rational curve in $\C^n$ and let $\mu$ be a positive Borel measure without point masses, supported by a compact subset of $\VV$. Then
$P^2(\mu) \neq L^2(\mu)$ if and only if there are analytic $P^2(\mu)$-bounded point evaluations.
\end{thm}

\begin{proof} Let $R = (r_1, r_2, \ldots, r_n)$ be an $n$-tuple of rational functions which properly parametrizes the affine curve $\VV$. That is, denoting by $S \subset \C$ the poles of $R$, the holomorphic map
$$ R : \C \setminus S \longrightarrow \VV $$
is one to one, except finitely many points, and it covers $\VV$ except finitely many points. We refer to Section 4.4.2 in \cite{SWP} for terminology and basic results. Let $\rho$ denote a sufficiently large radius, so that the support of the measure $\mu$ is contained in the ball $B(0,\rho)$. The pull-back 
$$ U = R^{-1} B(0,\rho)$$
is an open subset of $\C$, of finite connectivity, with piece-wise smooth boundary. In particular we can assume that every connected component of the complement of $U$ has positive diameter.

The restricted analytic map
$$ R : U \longrightarrow B(0,\rho)$$
has finite fibres, hence it is proper. Grauert's finiteness theorem implies that the direct image sheaf $R_\ast {\mathcal O}_U$ is coherent and
$$ R_\ast  {\mathcal O}_U (B(0,\rho)) = {\mathcal O}(U).$$
See again Theorem I.1.5 in \cite{GRtss}. As in the previous section, the coherence of $R_\ast {\mathcal O}_U$ and the injectivity of $R$ modulo a finite set
imply that
$$ \dim {\mathcal O}(U) / R^\ast {\mathcal O}(B(0,\rho)) < \infty.$$
Next we define the pull-back measure $\nu$ on $U$ as in the previous proof:
$$ \int \phi\,  d\mu = \int \phi \circ R \,d\nu,$$
for every continuous function $\phi : B(0,\rho) \longrightarrow \C$. 

Let $R^2(U,\nu)$ denote the closure in $L^2(\nu)$, of rational functions with poles on
the complement of $U$. Runge's approximation theorem implies that $R^2(U,\nu)$ is also the closure of the algebra ${\mathcal O}(U)$ in $L^2(\nu)$.
The counterpart of Proposition \ref{BPE} has the same proof:
\bigskip

{\it There exist analytic bounded point evaluations with respect to $P^2(\mu)$ if and only if there exists analytic bounded point evaluations with respect to  $R^2(U,\nu)$.}
\bigskip

Theorem 1 in \cite{Brennan-2008} asserts, under the positive diameter assumption of the connected components of $U$ and the lack of point masses, that $R^2(U,\nu) \neq L^2(\nu)$ if and only if there exist analytic bounded point evaluations with respect to $R^2(U,\nu)$. 
\end{proof}

\begin{cor}
A commuting subnormal tuple with Taylor's joint spectrum contained in a rational curve admits joint invariant subspaces.
\end{cor}
\begin{proof}
For a commuting subnormal tuple with Taylor's joint spectrum contained in a rational curve, it is enough to show the same for cyclic subnormal tuple of operators, in fact, by  \cite[Remark 2.2]{CS}, it suffices to show the same for the tuple of multiplication operator $M_z = (M_{z_1}, \ldots, M_{z_n})$ by the coordinate functions on $P^2(\mu)$. From Theorem \ref{rat}, it follows that if $\beta = (\beta_1, \ldots, \beta_n)$ is a bounded point evaluation point, then $\cap\ker (M_{z_i} - \beta_i)^*$ is non-trivial and hence the closure of $\sum (z_i - \beta_i)P^2(\mu)$ is a non-trivial invariant subspace for $M_z$.
\end{proof}

\section{Concluding remarks}

\subsection{Analytic coordinate charts} The above proof carries verbatim on an algebraic curve $\VV$, in case the measure $\mu$ is compactly supported in a coordinate chart $V \subset \VV$. More specifically, assuming the existence of a bi-holomorphic map
$$ f : U \longrightarrow V$$
where $U \subset \C$ is an open set with complement $\C \setminus U$ consisting of finitely many connected sets of positive diameter.

\subsection{Resolution of singularities} Uniform approximation by analytic functions defined on an open Riemann surface is much better understood, with definitive results generalizing Runge's Theorem or even Mergelyan's Theorem, see for instance  \cite{Bishop,Scheinberg}.
 
 Let $\VV$ be an affine algebraic curve. The well known desingularization procedure provides a birational transform 
 $p: X \longrightarrow \VV$, where $X$ is an open Riemann surface, see for instance \cite{Kollar}. Our main proof carries without major adaptation to the map $p$, raising the following intriguing question.
 \bigskip
 
 {\noindent \bf Open Problem.} {\it Let $X$ be an open Riemann surface of finite genus, and let $\mu$ be a positive Borel measure on $X$ without point masses. 
 Let $U \subset X$ be an open, relatively compact subset of $X$, with finitely many components of $X \setminus U$, none reduced to a point.
 Assume the closed support of the measure $\mu$ is contained in $U$. 
 Then analytic functions ${\mathcal O}(U)$ are dense in $L^2(\mu)$ if and only if there are no corresponding bounded analytic point evaluations.}
\bigskip
 
The precautions in stating the above question with specific topological constraints on the open set $U$ are resonant to Brennan's main result 
proved in the complex plane $X = \C$, \cite{Brennan-2008}. A simple application of the much stronger, uniform approximation results, can be formulated as follows.

\begin{rem}
Let $X$ be an open Riemann surface. Assume that the measure $\mu$ is supported by a piecewise smooth curve $\Gamma \subset X$ with the property that the complement
$ X \setminus \Gamma$ is connected. Then $P^2(\mu) = L^2(\mu)$. 
\end{rem}

Simply because ${\mathcal O}(X)$ is dense in the space of continuous functions $C(\Gamma)$, hence in $L^2(\mu)$, according to Scheinberg's master theorem \cite{Scheinberg}. The situation of an elliptic curve $X$, with $\Gamma$ as one of the generating cycles of homology of the projective completion of $X$, is notable in this respect.

\subsection{Generic linear projections} Let $\VV \subset \C^n$ be an affine algebraic curve and let $\mu$ be a positive Borel measure supported by $\VV$.
If there are $L^2(\mu)$-bounded point evaluations for polynomials, then any generic linear map 
$$ L : \C^n \longrightarrow \C$$
will detect them. Indeed, Thomson's Theorem applied to the measure $L_\ast \mu$ on $\C$ and the fact that $L$ is an open map, except the exceptional case when a component of $\VV$ is contained in the fibre of $L$, imply
$$ |p(L(\lambda)| \leq C \| p \circ L \|_{2, \mu} = C \| p \|_{2, L_\ast \mu}.$$
A simple example shows that the converse does not hold.

Indeed, let $\VV$ be the curve in $\C^2$ given by the equation $z_1 z_2 = 1$, and let $\mu$ be the positive measure
$$ \int p(z_1, z_2) d\mu = \int_{-\pi}^\pi p(e^{it},e^{-it}) dt, \ \ p \in \C[z_1,z_2].$$
If $(\lambda, \frac{1}{\lambda})$ were an analytic bounded point for this measure, by taking polynomial in $z_1$, or $z_2$, and say
$|\lambda|>1$, one would
obtain the impossible estimate
$$ |q(\lambda)|^2 \leq C \int_{-\pi}^\pi |q(e^{it})|^2 dt, \ \ q \in \C[z].$$
As a matter of fact, one finds easily the isometric identification $P^2(\mu) \equiv L^2(\T, d\theta)$, where $\T$ is the unit circle. The  pull-back map $\phi^*: f\mapsto f\circ \phi$ from $L^2(\mu)$ to $L^2(\mathbb T, d\theta)$, where $\phi: z\mapsto (z, \frac{1}{z})$ from $\mathbb T$ into $\VV$,  gives the isometric identification, since $\mu = \phi_*(d\theta)$ and $\phi^*\mathbb C[z_1, z_2] = \mathbb C[z] + \C[\frac{1}{z}]$ is dense in $L^2(\mathbb T, d\theta)$. Hence, the space
$P^2(\mu)$ does not carry bounded point evaluations.

On the other hand, considering the linear map $L(z_1,z_2) = z_1 + a z_2$, one finds that for all complex numbers $a$ of modulus different than $1$,
the measure $L_\ast \mu$ is integration along an ellipse, against a positive weight times arc length. Hence the space $P^2(L_\ast \mu)$ does admit bounded analytic point evaluations. In particular $P^2(L_\ast \mu) = H^2(\mathbb D)$ for $a=0$.

\end{document}